\documentclass[11pt]{amsart}

\usepackage[arrow, matrix, curve]{xy}
\usepackage{amscd}
\usepackage{amssymb}
\usepackage{amsthm}
\usepackage{amsmath}
\usepackage{mathrsfs}

\usepackage{stmaryrd}

\newtheorem{theorem}{Theorem}[section]

\theoremstyle{plain}
\newtheorem{definition}[theorem]{Definition}

\newtheorem{prop}[theorem]{Proposition}
\newtheorem{cor}[theorem]{Corollary}
\newtheorem{lemma}[theorem]{Lemma}
\newtheorem{thm}[theorem]{Theorem}

\theoremstyle{definition}

\newtheorem{notation}{Notation}

\theoremstyle{remark}
\newtheorem{remark}[theorem]{Remark}
\newcommand{\OA}{\mathcal{O}_{A}}
\newcommand{\OC}{\mathcal{O}}

\newcommand{\A}{\mathscr{A}}
\newcommand{\M}{\mathscr{M}}
\newcommand{\LL}{\mathcal{L}}
\newcommand{\V}{\mathcal{V}}
\newcommand{\E}{\mathcal{E}}
\newcommand{\G}{\mathcal{G}}
\newcommand{\ZZ}{\mathcal{Z}}
\newcommand{\PP}{\mathcal{P}}
\newcommand{\W}{\mathcal{W}}

\newcommand{\N}{\mathbb{N}}
\newcommand{\Z}{\mathbb{Z}}

\newcommand{\C}{\mathbb{C}}
\newcommand{\K}{\mathbb{K}}
\newcommand{\VV}{\mathbb{V}}

\newcommand{\g}{\mathfrak{g}}
\newcommand{\lb}{\mathfrak{b}}
\newcommand{\h}{\mathfrak{h}}
\newcommand{\p}{\mathfrak{p}}

\newcommand{\lam}{{\lambda}}

\newcommand{\Hom}{\mathrm{Hom}}
\newcommand{\End}{\mathrm{End}}
\newcommand{\Ext}{\mathrm{Ext}}

\numberwithin{equation}{section}

\begin{document}

\title{Tilting modules in category $\OC$ and sheaves on moment graphs}

\author{Johannes K\"ubel*}
\address{Department of Mathematics, University of Erlangen, Germany}
\curraddr{Cauerstr. 11, 91058 Erlangen, Germany}
\email{kuebel@mi.uni-erlangen.de}
\thanks{*supported by the DFG priority program 1388}


\date{\today}


\begin{abstract}
We describe tilting modules of the deformed category $\OC$ over a semisimple Lie algebra as certain sheaves on a moment graph associated to the corresponding block of $\OC$. We prove that they map to Braden-MacPherson sheaves constructed along the reversed Bruhat order under Fiebig's localization functor. By this means, we get character formulas for tilting modules and explain how Soergel's result about the Andersen filtration gives a Koszul dual proof of the semisimplicity of subquotients of the Jantzen filtration.
\end{abstract}

\maketitle

\section{Introduction}

Let $\g \supset \lb \supset \h$ be a complex semisimple Lie algebra with a Borel and a Cartan subalgebra. Let $A$ be the localization of the symmetric algebra $S=S(\h)$ at the maximal ideal of zero. The deformed category $\OC_A$ is the full subcategory of $\g$-$A$-bimodules that are finitely generated over $\g \otimes_\C A$, semisimple over $\h$ and locally finite over $\lb$. $\OC_A$ decomposes into blocks which are parameterized by antidominant weights. For a given antidominant weight $\lam \in \h^*$ the weights involved in the corresponding block are given by the orbit $\W_\lam \cdot \lam$ of $\lam$ under the dot-action of the integral Weyl group corresponding to $\lam$. This combinatorial data defines a graph with $\W_\lam \cdot \lam$ the set of vertices being partially ordered by the Bruhat order on $\W_\lam$ divided by the stabilizer of $\lam$. Two different vertices are linked by an edge if there is a reflection of $\W_\lam$ mapping one vertex to the other. In addition, every edge has a labeling given by the coroot corresponding to the positive root of the according reflection. Denote by $\mathcal{M}_{A,\lam}$ the subcategory of the block corresponding to the antidominant weight $\lam$ consisting of modules which have a Verma flag, i.e., a filtration with subquotients isomorphic to deformed Verma modules. Now the usual duality on category $\OC$ extends to the deformed version $\OC_A$. The modules which are self-dual and admit a Verma flag are called deformed tilting modules. The indecomposable tilting modules are parameterized by their highest weight and $\mathcal{M}_{A,\lam}$ contains those with a highest weight lying in the orbit $\W_\lam \cdot \lam$.\\
While Fiebig shows that indecomposable deformed projective modules of $\mathcal{M}_{A,\lam}$ correspond to Braden-MacPherson sheaves constructed along the Bruhat order of the associated moment graph, we prove in a very similar way that indecomposable deformed tilting modules correspond to Braden-MacPherson sheaves constructed along the reversed order on the moment graph. This approach implies a character formula for tilting modules which was already discovered in \cite{A}. There, Soergel uses a tilting equivalence to trace back character formulas for tilting modules to the known ones for projective modules. Our approach, however, doesn't use the tilting functor but has the disadvantage that it doesn't generalize to Kac-Moody algebras without using the tilting functor.\\
Another application of our result about tilting modules as sheaves on a moment graph is the connection between the Andersen and Jantzen filtration. The Jantzen filtration on a Verma module induces a filtration on the space of homomorphisms from a projective to the Verma module. The Andersen filtration on the space of homomorphisms from a Verma module to a tilting module is constructed in a very similar way as the Jantzen filtration. In \cite{11} we already proved that there is an isomorphism between both spaces of homomorphisms which interchanges the mentioned filtrations.\\
In this paper we describe both Hom-spaces on the level of sheaves on moment graphs. Since the construction of those involves the symmetric algebra $S(\h)$ we discover an inherited grading on both Hom-spaces. Now the advantage of this approach compared to \cite{11} is that we are able to construct an isomorphism which respects the grading and lifts to an isomorphism on the Hom-spaces which also interchanges both filtrations. In \cite{14} it is proved that the Andersen filtration coincides with the grading filtration on this Hom-space. Soergel's approach, however, is Koszul dual to \cite{1} and in combination with our result, leads to another proof of the semisimplicity of the Jantzen filtration layers.
\section{Preliminaries}

In this section we repeat results of \cite{7} and \cite{3} about certain sheaves on moment graphs. We mostly follow the lecture notes \cite{10} and \cite{15} which are more introductory to this subject.

\subsection{Moment graphs}

For a vector space $V$ we denote by $S:= S(V)$ the symmetric algebra of $V$ with the usual grading doubled, i.e., $\mathrm{deg} V=2$. Let $\V$ be the set of vertices and $\E$ the set of edges of a finite graph $(\V,\E)$. I.e., $\V$ is a finite set and $\E \subset \PP(\V)$ a subset of the power set of $\V$ with the following property:\\
If $E$ is an element of $\E$, then the cardinality of $E$ is two.

\begin{definition}
An unordered $V$-moment graph $\G=(\V,\E, \alpha)$ is a finite graph $(\V,\E)$ without loops and double edges, which is equipped with a map $\alpha: \E \rightarrow \mathbb{P}(V)$ that associates to any edge $E$ a line $\alpha_E :=\alpha(E)$ in $V$.
\end{definition}

\begin{remark}
The subsets $Y \subset \V \cup \E$ with the property:
$$x \in Y \cap \V \Rightarrow \{ E \in \E \,| \, x \in E\} \subset Y$$
form the open sets of a topology on $\V \cup \E$. By this means, we view $\G$ as a topological space.
\end{remark}

\subsection{Sheaves on moment graphs}

Let $A$ be a commutative $S$-algebra. For $x \in \V$ define $x^{\circ}:= \{x\} \cup \{E \in \E \,| \, x\in E\}$. For a sheaf $\M$ of $A$-modules on the topological space $\G$ the stalks are given by\\
$\M_x = \M(x^{\circ})$ for $x\in \V$ and $\M_E = \M(\{E\})$ for $E \in \E$.\\
We denote by $\rho_{x,E}: \M_x \rightarrow \M_E$ the restriction map for $x \in E$. The sheaf $\M$ is uniquely determined by this data and we can define a sheaf of rings $\A$, namely the \textit{structure sheaf} of $\G$ over $A$, by setting
\begin{itemize}
\item $\A_x =A$  $\forall x \in \V$
\item $\A_E = A / \alpha_E A$  $\forall E \in \E$
\item $\rho_{x,E}: \A_x \rightarrow \A_E$ the quotient map for $x \in E$.
\end{itemize}

By \cite{3} Proposition 1.1, an $\A$-module $\M$ is characterized by a tuple \\
$(\{\M_x\},\{\M_E\},\{\rho_{x,E}\})$ with the properties
\begin{itemize}	
\item $\M_x$ is an $A$-module for any $x \in \V$
\item $\M_E$ is an $A$-module for all $E \in \E$ with $\alpha_E \M_E =0$
\item $\rho_{x,E}: \M_x \rightarrow \M_E$ is a homomorphism of $A$-modules for $x \in \V$, $E \in \E$ with $x \in E$.
\end{itemize}

\begin{remark}
In what follows, we will always work with this characterization of sheaves on the moment graph. If the $S$-algebra $A$ is $S$ itself, we consider all modules as graded $S$-modules and all maps between them as graded homomorphisms of degree zero.\\
To distinguish between the $S$-algebras we are working with we sometimes call the sheaf $\M$ an $A$-sheaf.
\end{remark}

\subsection{Global sections}

Now let $A$ be a localization of $S$ at a prime ideal $\p \subset S$. Denote by $\mathcal{SH}_A(\G)^{f}$ the subcategory of $\A$-modules, such that $\M_x$ is torsion free and finitely generated over $A$ for all $x \in \V$. We denote by $\ZZ = \ZZ_A (\G)$ the global sections $\Gamma(\A)$ of the structure sheaf and call it the \textit{structure algebra of $\G$ over $A$}. By \cite{7} section 2.5. we get $\ZZ :=\ZZ_A(\G) = \{(z_x) \in \prod_{x \in \V} A \, | \, (z_x \equiv z_y \, \mathrm{mod} \, \alpha_E) \, \mathrm{for} \, \{x,y\}=E\}$.

\begin{remark}
In case $A=S$, $\ZZ_S(\G)$ carries a grading induced by $S$. In this case we consider all $\ZZ$-modules as graded modules.
\end{remark}

The functor of global sections

\begin{equation*}
\Gamma: \A- \mathrm{mod} \longrightarrow \ZZ-\mathrm{mod}
\end{equation*}
has a left adjoint, namely the localization functor $\LL$. Denote by $\ZZ_A-\mathrm{mod}^f$ the subcategory of $\ZZ$-modules that are finitely generated and torsion free over $A$. 

\begin{lemma}(\cite{7}, Proposition 3.5.) \label{adj}
The functors $\Gamma$ and $\LL$ induce a pair of adjoint functors

\[
\begin{xy}
\xymatrix{
  \mathcal{SH}_A(\G)^f   \ar@<2pt>[r]^{\Gamma}  &  \ZZ_A-\mathrm{mod}^f  \ar@<2pt>[l]^{\LL}  \\
} 
\end{xy}
\]
and the canonical maps $\Gamma(\M) \rightarrow \Gamma \LL \Gamma (\M)$ and $\LL(M) \rightarrow \LL \Gamma \LL (M)$ are isomorphisms.
\end{lemma}

\begin{remark}
Lemma \ref{adj} implies that we get a pair of mutually inverse equivalences between the images of both functors

\[
\begin{xy}
\xymatrix{
   \LL(\ZZ_A-\mathrm{mod}^f )  \ar@<2pt>[r]^{\Gamma}  &  \Gamma(\mathcal{SH}_A(\G)^f) \ar@<2pt>[l]^{\LL}  \\
} 
\end{xy}
\]
\end{remark}

We now follow \cite{10} to give a concrete description of $\LL$. Let $M \in \ZZ_A-\mathrm{mod}^f$ and denote by $Q$ the quotient field of $A$. Since $M$ is torsion free over $A$ we get an inclusion $M \hookrightarrow M \otimes_A Q$.\\
Let $\sum_{x \in \V} {e_x} = 1 \otimes 1 \in \ZZ \otimes_A Q \cong \prod_{x \in \V} Q$ be a decomposition of $1 \in \ZZ \otimes_A Q$ into idempotents. For $x \in \V$ set 
$$\LL(M)_x = e_xM \subset M \otimes_A Q$$
For an edge $E=\{x,y\}$ with $\alpha:= \alpha(E)$ we set
$$M(E) = (e_x + e_y) M + \alpha e_x M \subset e_x(M\otimes_A Q) \oplus e_y (M \otimes_A Q)$$
and form the push-out diagram

\begin{eqnarray*}
	\begin{CD}
   M(E)  @>\pi_x >> \LL(M)_x\\
   @VV\pi_y V @VV\rho_{x,E} V\\
   \LL(M)_y @>\rho_{y,E} >> \LL(M)_E
	\end{CD}
\end{eqnarray*}

where $\pi_x$, $\pi_y$ are defined by $\pi_x(z)=e_x z$ and $\pi_y(z) = e_y z$. 

This gives the sought after stalk $\LL(M)_E$ with restriction maps $\rho_{x,E}$, $\rho_{y,E}$ coming from the push-out diagram.

\subsection{Sheaves on ordered moment graphs}

\begin{definition}
An ordered moment graph $\G=(\V,\E,\alpha, \leq)$ is a moment graph $(\V,\E,\alpha)$ with a partial order $\leq$ on the set $\V$ of vertices, such that for any $E = \{x,y\}$ the vertices $x, y \in \V$ are comparable.
\end{definition}

\begin{definition}
An F-open subgraph $\mathcal{H} = (\V',\E', \alpha',\leq')$ of $\G$ is a subgraph with $\alpha'$ and $\leq'$ the restrictions of $\alpha$ and $\leq$, respectively, such that
\begin{itemize}
\item If $E= \{x,y\}$ and $x,y \in \V'$, then $E \in \E'$
\item If $x \in \V'$ and $y \in \V$ with $y \leq x$, then $y \in \V'$
\end{itemize}
\end{definition}

\begin{definition}
An $A$-sheaf $M$ on $\G$ is called F-flabby if for any F-open subgraph $\mathcal{H}$ of $\G$ the restriction map $\Gamma(\M) \rightarrow \Gamma(\mathcal{H}, \M)$ is surjective, where $\Gamma(\mathcal{H}, \M):= \{(m_x) \in \prod_{x \in \V'} \M_x \, | \, \rho_{x,E} (m_x) = \rho_{y,E}(m_y) \,\, \mathrm{for}\,\, E=\{x,y\}\}$
\end{definition}

\begin{definition}
A moment graph $(\G, \alpha, \leq)$ has the GKM-property if for every pair $E, E' \in \E$ with $E \neq E'$ and $E \cap E' \neq \emptyset$ we have $\alpha(E) \neq \alpha(E')$.
\end{definition}

Recall that we decomposed $1 \in \ZZ \otimes_A Q \cong \prod_{x \in \V} Q$ into idempotents $1= \sum_{x \in \V}e_x$. For an F-open subset $\mathcal{H}$ of $\G$ with vertices $\V'$, we set $e_{\mathcal{H}} := \sum_{x \in \mathcal{V}'} e_x$.

\begin{prop}[\cite{10}, Proposition 3.14] \label{B}
Suppose that $\G$ is a GKM-graph. Let $M$ be a finitely generated $\ZZ$-module that is torsion free over A. Suppose in addition, that $e_{\mathcal{H}}M$ is a reflexive $A$-module for any F-open subgraph $\mathcal{H}$ of $\G$. Then $\LL(M)$ is F-flabby on $\G$ and we have an isomorphism
$$e_{\mathcal{H}}M \stackrel{\sim}{\longrightarrow} \Gamma({\mathcal{H}}, \LL(M))$$
\end{prop}

\subsection{Braden-MacPherson sheaves}

In this section we will repeat the notion of F-projective sheaves on a moment graph and introduce Braden-MacPherson sheaves which form the indecomposable F-projective sheaves.
\begin{definition}
A sheaf $\M$ on $\G$ is \textit{generated by global sections} if the map $\Gamma(\M) \rightarrow \M_x$ is surjective for every $x \in \V$.
\end{definition}

\begin{notation} \label{not1}
For any $x \in \V$ set $D_x := \{E \in \E \,| \, E=\{x,y\}$, $y \in \V$, $y\leq x\}$ and $U_x := \{E \in \E \,|\, E=\{x,y\}$, $y\in \V$, $x \leq y\}$.
\end{notation}

\begin{definition}
An $A$-sheaf $\mathscr{P}$ on $\G$ is called F-projective if
\begin{itemize}
\item $\mathscr{P}$ is F-flabby and generated by global sections 
\item Each $\mathscr{P}_x$ with $x \in \V$ is a free (graded free for $A=S$) $A$-module 
\item Any $\rho_{x,E}$ with $x \in \V$, $E \in U_x$ induces an isomorphism $\mathscr{P}_x / \alpha_E\mathscr{P}_x \rightarrow \mathscr{P}_E$ of (graded) $A$-modules.
\end{itemize}
\end{definition}

Next, we cite some results about Braden-MacPherson (BMP) sheaves from \cite{B} and \cite{10}. For this we take $A=S_{(0)}$ to be the localization of $S$ at the maximal ideal generated by $V$. 

\begin{thm}[\cite{10}, section 3.5 and \cite{B}, Theorem 6.3]
\begin{enumerate}
\item For any $x \in \V$ there is an up to isomorphism unique graded $S$-sheaf $\mathscr{B}(x)$ on $\G$ with the following
properties:
	\begin{itemize}
	\item $\mathscr{B}(x)$ is F-projective
	\item $\mathscr{B}(x)$ is indecomposable (even as a non-graded sheaf)
	\item $\mathscr{B}(x)_x \cong S$ and $\mathscr{B}(x)_y = 0$ unless $ x \leq y$
	\end{itemize}
	\item Let $\mathscr{P}$ be an F-projective $A$-sheaf of finite type on $\G$. Then there exists an isomorphism of $A$-sheaves
	$$\mathscr{P} \cong \mathscr{B}(z_1) \otimes_S A \oplus ... \oplus \mathscr{B}(z_n) \otimes_S A$$
	with suitable vertices $z_1,..., z_n$.
	\item Let $\mathscr{P}$ be a graded F-projective $S$-sheaf of finite type on $\G$. Then there exists an isomorphism of graded $S$-sheaves
	$$\mathscr{P} \cong \mathscr{B}(z_1)[l_1] \oplus ... \oplus \mathscr{B}(z_n)[l_n]$$
	with suitable vertices $z_1,..., z_n$ and suitable shifts $l_1,...,l_n$.
\end{enumerate}	
\end{thm}

\section{Deformed category $\OC$}

In this section we recall results about the deformed category $\OC$ of a semisimple complex Lie algebra $\g$ with Borel $\lb$ and Cartan $\h$, which one can find in \cite{5}, \cite{14} and \cite{11}. Let $S$ denote the universal enveloping algebra of the Cartan $\h$ which is equal to the ring of polynomial functions $\C[\h^{*}]$.
We call a commutative, associative, noetherian, unital, local $S$-algebra $A$ with structure morphism $\tau: S \rightarrow A$ a \textit{local deformation algebra}.\\

Let $A$ be a local deformation algebra with structure morphism $\tau: S\rightarrow A$ and let $M \in \g$-mod-$A$. For $\lam \in \h^{*}$ we set
$$M_\lam = \{m \in M| hm=(\lam + \tau)(h)m \, \, \forall h \in \h\}$$ 
where $(\lam+\tau)(h)$ is meant to be an element of $A$. We call the $A$-submodule $M_\lam$ the \textit{deformed $\lam$-weight space} of $M$.\\

The \textit{deformed category} $\OA$ is the full subcategory of all bimodules $M \in \g$-mod-$A$ with the properties
\begin{itemize}
\item $M= \bigoplus\limits_{\lam \in \h^{*}} M_\lam$,
\item for every $m \in M$ the $\lb$-$A$-sub-bimodule generated by $m$ is finitely generated as an $A$-module,
\item $M$ is finitely generated as a $\g$-$A$-bimodule.
\end{itemize}

Taking $A=S/S\h \cong \C$, $\OA$ is just the usual BGG-category $\OC$.\\
For $\lam \in \h^{*}$ the \textit{deformed Verma module} is defined by
$$\Delta_A(\lam) = U(\g) \otimes_{U(\lb)} A_\lam$$
where $A_\lam$ denotes the $U(\lb)$-$A$-bimodule $A$ with $\lb$-structure given by the composition $U(\lb) \rightarrow S \stackrel{\lam + \tau}{\longrightarrow} A$.\\

The Lie algebra $\g$ possesses an involutive anti-automorphism $\sigma:\g \rightarrow \g$ with $\sigma|_{\h} = -\mathrm{id}$. This gives the $A$-module $\Hom_A(M,A)$ a $\sigma$-twisted $\g$-module structure. Denoting by $dM$ the sum of all deformed weight spaces in $\Hom_A(M,A)$, we get a functor

$$d=d_\sigma : \OA \longrightarrow \OA$$

which is a duality on $\g$-$A$-bimodules which are free over $A$. We now set $\nabla_A(\lam)=d\Delta_A(\lam)$ for $\lam \in \h^{*}$ and call this the \textit{deformed nabla module}.

\begin{prop}[\cite{14}, Proposition 2.12.]\label{Soe1}
\begin{enumerate}
\item For all $\lam$ the restriction to the deformed weight space of $\lam$ together with the two canonical identifications $\Delta_A(\lam)_\lam \stackrel{\sim}{\rightarrow} A$ and $\nabla_A(\lam)_\lam \stackrel{\sim}{\rightarrow} A$ induces an  isomorphism
$$\Hom_{\OA}(\Delta_A(\lam),\nabla_A(\lam)) \stackrel{\sim}{\longrightarrow}A$$

\item For $\lam \neq \mu$ in $\h^{*}$ we have $\Hom_{\OA}(\Delta_A(\lam),\nabla_A(\mu))=0$.

\item For all $\lam,\mu \in \h^{*}$ we have $\Ext^{1}_{\OA}(\Delta_A(\lam),\nabla_A(\mu))=0$.
\end{enumerate}
\end{prop}

\begin{cor}[\cite{14}, Corollary 2.13.]\label{Soe2}
Let $M,N \in \OA$. If $M$ has a $\Delta_A$-flag and $N$ a $\nabla_A$-flag, then the space of homomorphisms $\Hom_{\OA}(M,N)$ is a finitely generated free $A$-module and for any homomorphism $A\rightarrow A'$ of local deformation algebras the obvious map defines an isomorphism
$$\Hom_{\OA}(M,N)\otimes_{A} A' \stackrel{\sim}{\longrightarrow}   \Hom_{\OC_{A'}}(M\otimes_{A} A',N\otimes_{A} A')$$
\end{cor}

\begin{proof}
This follows from Proposition \ref{Soe1} by induction on the length of the $\Delta_A$- and $\nabla_A$-flag.
\end{proof}

If $\mathfrak{m}\subset A$ is the unique maximal ideal in the local deformation algebra $A$ we set $\K=A/\mathfrak{m}A$ for its residue field.

\begin{thm}[\cite{5}, Propositions 2.1 and 2.6] \label{Fie1}

\begin{enumerate}
\item The base change $\cdot \otimes_A \K$ gives a bijection\\
\begin{displaymath}
		\begin{array}{ccc}
		
			\left\{\begin{array}{c}
        \textrm{simple isomorphism}\\
      	\textrm{classes of $\OA$}
    	\end{array}\right\}
		&
		\longleftrightarrow
		&
			\left\{\begin{array}{c}
        \textrm{simple isomorphism}\\
      	\textrm{classes of $\OC_{\K}$}
    	\end{array}\right\}
		\end{array}
	\end{displaymath}

\item The base change $\cdot \otimes_A \K$ gives a bijection\\	
\begin{displaymath}
		\begin{array}{ccc}
		
			\left\{\begin{array}{c}
        \textrm{projective isomorphism}\\
      	\textrm{classes of $\OA$}
    	\end{array}\right\}
		&
		\longleftrightarrow
		&
			\left\{\begin{array}{c}
        \textrm{projective isomorphism}\\
      	\textrm{classes of $\OC_{\K}$}
    	\end{array}\right\}
		\end{array}
	\end{displaymath}
\end{enumerate}
\end{thm}

The category $\OC_{\K}$ is a direct summand of the category $\OC$ over the Lie algebra $\g \otimes \K$. It consists of all objects whose weights lie in the complex affine subspace $\tau + \h^{*} = \tau + \Hom_\C(\h,\C) \subset \Hom_\K(\h\otimes \K,\K)$ for $\tau$ the restriction to $\h$ of the map that makes $\K$ to an $S$-algebra. Thus the simple objects of $\OA$ are parameterized by their highest weight in $\h^{*}$. Denote by $L_A(\lam)$ the simple object with highest weight $\lam$. We also use the usual partial order on $\h^{*}$ to partially order $\tau + \h^{*}$.

\begin{thm}[\cite{5}, Propositions 2.4 and Theorem 2.7] \label{Fie2}
Let $A$ be a local deformation algebra and $\K$ its residue field. Let $L_A(\lam)$ be a simple object in $\OA$.

\begin{enumerate}
\item There is a projective cover $P_A(\lam)$ of $L_A(\lam)$ in $\OA$ and every projective object in $\OA$ is isomorphic to a direct sum of projective covers.
\item $P_A(\lam)$ has a Verma flag, i.e., a finite filtration with subquotients isomorphic to Verma modules, and for the multiplicities we have the BGG-reciprocity formula 
$$(P_A(\lam):\Delta_A(\mu)) = [\Delta_{\K}(\mu):L_{\K}(\lam)]$$
for all Verma modules $\Delta_A(\mu)$ in $\OA$.
\item Let $A \rightarrow A'$ be a homomorphism of local deformation algebras and $P$ projective in $\OA$. Then $P\otimes_A A'$ is projective in $\OC_{A'}$ and the natural transformation
$$\Hom_{\OA}(P,\cdot) \otimes_A A' \longrightarrow \Hom_{\OC_{A'}}(P \otimes_A A', \cdot \otimes_A A')$$
is an isomorphism of functors from $\OA$ to $A'$-mod.
\end{enumerate}
\end{thm}

\subsection{Block decomposition}

Let $A$ again denote a local deformation algebra and $\K$ its residue field.

\begin{definition}
Let $\sim_A$ be the equivalence relation on $\h^{*}$ generated by $\lam \sim_A \mu$ if $[\Delta_{\K}(\lam):L_{\K}(\mu)] \neq 0$.
\end{definition}

\begin{definition}
Let $\Lambda \in \h^{*}/\sim_{A}$ be an equivalence class. Let $\OC_{A,\Lambda}$ be the full subcategory of $\OA$ consisting of all modules $M$ such that every highest weight of a subquotient of $M$ lies in $\Lambda$.
\end{definition}

\begin{prop}[\cite{5}, Proposition 2.8] (Block decomposition)\label{Fie3}
The functor
\begin{eqnarray*}
		\begin{array}{ccc}
		\bigoplus\limits_{\Lambda \in \h^{*}/\sim_A} \OC_{A,\Lambda}  &\longrightarrow&  \OA\\
		(M_\Lambda)_{\Lambda \in \h^{*}/\sim_A}& \longmapsto & \bigoplus\limits_{\Lambda \in \h^{*}/\sim_A} M_\Lambda
		
		\end{array}
	\end{eqnarray*}
is an equivalence of categories.
\end{prop}

\begin{remark}
For $R=S_{(0)}$ the localization of $S$ at the maximal ideal generated by $\h$,
the block decomposition of $\OC_R$ corresponds to the block decomposition of the BGG-category $\OC$ over $\g$.\\
\end{remark}

Let $\tau:S \rightarrow \K$ be the induced map that makes $\K$ into an $S$-algebra. Restricting to $\h$ and extending with $\K$ yields a $\K$-linear map $\h \otimes \K \rightarrow \K$ which we will also call $\tau$. Let $\mathcal{R} \supset \mathcal{R}^{+}$ be the root system with positive roots according to our data $\g\supset \lb \supset \h$. For $\lam \in \h_{\K}^{*}=\Hom_\K(\h \otimes \K,\K)$ and $\check{\alpha} \in \h$ the dual root of a root $\alpha \in \mathcal{R}$ we set $\left\langle \lam, \check{\alpha}\right\rangle_{\K} = \lam(\check{\alpha}) \in \K$. Let $\mathcal{W}$ be the Weyl group of $(\g,\h)$ and denote by $s_\alpha$ the reflection corresponding to $\alpha \in \mathcal{R}$.

\begin{definition}
For $\mathcal{R}$ the root system of $\g$ and $\Lambda \in \h^{*}/\sim_A$ we define 
$$\mathcal{R}_A(\Lambda)=\{\alpha \in \mathcal{R} | \left\langle \lam + \tau, \check{\alpha}\right\rangle_{\K} \in \Z \subset \K \text{ for some } \lam \in \Lambda\}$$
and call it the integral roots corresponding to $\Lambda$. Let $\mathcal{R}_A^{+}(\Lambda)$ denote the positive roots in $\mathcal{R}_A(\Lambda)$ and set $$\mathcal{W}_A(\Lambda)=\langle \{s_\alpha \in \mathcal{W} | \alpha \in \mathcal{R}_A^{+}(\Lambda)\}\rangle \subset \mathcal{W}$$
We call it the integral Weyl group with respect to $\Lambda$.
\end{definition}
From \cite{5} Corollary 3.3 it follows that
$$\Lambda=\mathcal{W}_A(\Lambda)\cdot \lam \text{  for any  } \lam \in \Lambda$$
where we denote by $\cdot$ the $\rho$-shifted dot-action of the Weyl group.\\
Since most of the following constructions commute with base change, we are particularly interested in the case when $A=R_{\p}$ is a localization of $R$ at a prime ideal $\p$ of height one. The functor $\cdot \otimes_R R_{\p}$ will split the deformed category $\OA$ into generic and subgeneric blocks:

\begin{lemma}[\cite{6}, Lemma 3] \label{Fie4}
Let $\Lambda \in \h^{*}/\sim_{R}$ and let $\p \in R$ be a prime ideal.
\begin{enumerate}
\item If $\check{\alpha} \notin \p$ for all roots $\alpha \in \mathcal{R}_{R}(\Lambda)$, then $\Lambda$ splits under $\sim_{R_\p}$ into generic equivalence classes.
\item If $\p = R\check{\alpha}$ for a root $\alpha \in \mathcal{R}_{R}(\Lambda)$, then $\Lambda$ splits under $\sim_{R_\p}$ into subgeneric equivalence classes of the form $\{\lam,s_{\alpha}\cdot \lam\}$.
\end{enumerate}
\end{lemma}
We recall that we denote by $P_A(\lam)$ the projective cover of the simple object $L_A(\lam)$. It is indecomposable and up to isomorphism uniquely determined.
For an equivalence class $\Lambda \in \h^{*}/\sim_A$ which contains $\lam$ and is generic, i.e., $\Lambda=\{\lam\}$, we get $P_A(\lam)=\Delta_A(\lam)$. If $\Lambda=\{\lam,\mu\}$ and $\mu< \lam$, we have $P_A(\lam)=\Delta_A(\lam)$ and there is a non-split short exact sequence in $\OA$
$$0\rightarrow \Delta_A(\lam) \rightarrow P_A(\mu) \rightarrow \Delta_A(\mu)\rightarrow 0$$ 
In this case, every endomorphism $f: P_A(\mu) \rightarrow P_A(\mu)$ maps $\Delta_A(\lam)$ to $\Delta_A(\lam)$ since $\lam>\mu$. So $f$ induces a commutative diagram 
\begin{eqnarray*}
\begin{CD}
   0   @>>> \Delta_{A}(\lam) @>>> P_{A}(\mu) @>>> \Delta_{A}(\mu) @>>> 0\\
   @VVV @V f_\lam VV @VV f V @VV f_\mu V @VVV\\
   0   @>>> \Delta_{A}(\lam) @>>> P_{A}(\mu) @>>> \Delta_{A}(\mu) @>>> 0 
\end{CD}\end{eqnarray*}
Since endomorphisms of Verma modules correspond to elements of $A$, we get a map
\begin{displaymath}
		\begin{array}{ccc}
		\chi: \End_{\OC_{A}}(P_{A}(\mu))& \longrightarrow & A \oplus A \\
		f & \longmapsto &(f_\lam , f_\mu)  
		\end{array}
\end{displaymath}

For $\p=R\check{\alpha}$ we define $R_\alpha :=R_\p$ for the localization of $R$ at the prime ideal $\p$.

\begin{prop}[\cite{5}, Corollary 3.5] \label{Fie5}
Let $\Lambda \in \h^{*}/\sim_{R_\alpha}$.
If $\Lambda=\{\lam,\mu\}$ and $\lam=s_{\alpha}\cdot \mu >\mu$, the map $\chi$ from above induces an isomorphism of $R_\alpha$-modules
$$\End_{\OC_{R_\alpha}}(P_{R_\alpha}(\mu)) \cong \left\{(t_\lam,t_\mu) \in {R_\alpha}\oplus {R_\alpha} \middle| t_\lam \equiv t_\mu \text{ mod } \check{\alpha}\right\}$$
\end{prop}

\subsection{Deformed tilting modules}

In this chapter, $A$ will be a localization of $R=S_{(0)}$ at a prime ideal $\p \subset R$ and $\K$ its residue field. Let $\lam \in \h^{*}$ be such that $\Delta_{\K}(\lam)$ is a simple object in $\OC_{\K}$. Thus, we have $\Delta_{\K}(\lam) \cong \nabla_{\K}(\lam)$ and the canonical inclusion $\Delta_A(\lam) \hookrightarrow \nabla_A(\lam)$ becomes an isomorphism after applying $\cdot \otimes_A \K$. So by Nakayama's lemma, we conclude that this inclusion is bijective.

\begin{definition}
By $\mathcal{K}_A$ we denote the full subcategory of $\OA$ which
\begin{enumerate}
\item includes the self-dual deformed Verma modules,
\item is stable under tensoring with finite dimensional $\g$-modules,
\item is stable under forming direct sums and summands. 
\end{enumerate}
\end{definition}

\begin{prop}(\cite{11}, Proposition 3.2.) \label{tilt1}
The base change $\cdot \otimes_A \K$ gives a bijection\\	
\begin{displaymath}
		\begin{array}{ccc}
		
			\left\{\begin{array}{c}
        \textrm{isomorphism classes}\\
      	\textrm{of $\mathcal{K}_A$}
    	\end{array}\right\}
		&
		\longleftrightarrow
		&
			\left\{\begin{array}{c}
        \textrm{isomorphism classes}\\
      	\textrm{of $\mathcal{K}_\K$}
    	\end{array}\right\}
		\end{array}
	\end{displaymath}
\end{prop}

\begin{remark}
For $A=S/S\h=\C$ the category $\mathcal{K}_A$ is just the usual subcategory of tilting modules of the category $\OC$ over $\g$.
The definition also implies that deformed tilting modules have a Verma flag and are self-dual. Furthermore, the indecomposable tilting modules are classified by their highest weight and we denote by $K_A(\lam)$ the deformed tilting module with highest weight $\lam \in \h^*$.
\end{remark}

\section{Tilting modules as sheaves on moment graphs}

In this section we repeat the connection between representation theory and sheaves on moment graphs via the structure functor $\VV$ as it is described in \cite{7}. We prove without using the tilting functor that tilting modules of a deformed block in category $\OC$ become certain BMP-sheaves on a certain moment graph associated to this block. As a corollary of this we get character formulas of tilting modules without using the tilting functor.

\subsection{The functor $\VV$}

Again, $A$ denotes a localization of $R$ at a prime ideal. We first want to get a functor from a block of deformed category $\OC$ to sheaves on a certain moment graph. Given an antidominant weight $\lam \in \h^*$ denote by $\Lambda \in \h^* /\sim_A$ its equivalence class. We now set $\mathcal{W}_\lam = \mathcal{W}_A (\Lambda)$ and $\mathcal{W}':= \mathrm{Stab}_{\mathcal{W}_\lam} (\lam)$. We then get a bijection
$$\mathcal{W}_\lam \cdot \lam \cong \mathcal{W}_\lam / \mathcal{W}'$$
Now we define the ordered $\h^*$-moment graph $\G=(\V,\E,\alpha, \leq)$ associated to the block $\mathcal{W}_\lam \cdot \lam$ by letting
 $$\V := \mathcal{W}_\lam / \mathcal{W}'$$
 and two different vertices $x,y \in \V$ are joined by an edge $E=\{x,y\}$ if there is a positive root $\alpha \in \mathcal{R}_A^+(\Lambda)$ with $x=s_\alpha \cdot y$. The labeling of $E$ is $\alpha(E)= \check{\alpha}$. For $w,w' \in \V$ we define the order $\leq$ by
$$w \leq w' \Leftrightarrow w \cdot \lam \leq w'\cdot \lam$$
This ordered moment graph has the GKM-property. Note that this order is not the Bruhat order in general. But for two adjacent vertices both orderings coincide.

\begin{thm}[\cite{5}, Theorem 3.6.]
Let $\ZZ= \ZZ_A (\G)$ be the structure algebra of the above moment graph. Then there is an isomorphism 
$$\ZZ \cong \End_{\OC_A}(P_A(\lam))$$
\end{thm}

We now get a functor                                                                                                                                                                                                                                                                                                                                                                                                                                                                                                                                                                                                   

$$\VV := \Hom_{\OC_A}(P_A(\lam), \cdot) : \OC_{A, \lam} \longrightarrow \ZZ-mod$$

\begin{definition}
Denote by $\mathscr{V}_A(w)$ the skyscraper sheaf on $\G$ at the vertex $w \in \V$, i.e., the $A$-sheaf with $\mathscr{V}_A(w)_w \cong A$ and whose stalks at every other vertex and at every edge are zero.
\end{definition}

\begin{prop}
\begin{enumerate}
\item Let $w \in \V$. Then $\LL(\VV \Delta_A(w \cdot \lam)) \cong \LL(\VV \nabla_A(w \cdot \lam)) \cong \mathscr{V}_A(w)$.

\item Let $M \in \OC_{A,\lam}$ admit a Verma- or a Nabla-flag. Then $\VV M$ is a free $A$-module of finite rank.
\end{enumerate}
\end{prop}

\begin{proof}
The proof is analogous to \cite{10} Proposition 4.13.
\end{proof}

\subsection{Deformed tilting modules and BMP-sheaves}\label{BMP}

Following \cite{10} section 4.14., we construct certain submodules and quotients of a module $M \in \OC_{A}$.
Let $D$ be a subset of $\h^*$ with the property:\\
If $\lam \in D$ and $\mu \in \h^*$ with $\mu \leq \lam$, then $\mu \in D$.\\
Set for $M \in \OC_A$
$$ O^D M:= \sum_{\mu \notin D} U(\g_A) M_\mu \, \, and \,\, M[D] := M/O^D M$$

\begin{prop} [\cite{10}, section 4.14.]´
	\begin{enumerate}
	\item If $M$ has a Verma flag, so do $O^D M$ and $M[D]$
	\item 
   \begin{displaymath}
		\begin{array}{ccc}
		
			O^D\Delta_A(\lam) \cong		 
					\left\{ \begin{array}{c}
        \Delta_A (\lam) \, \textrm{if}\, \lam \notin D\\
        0 \,\,\,\,\,\,\,\,\, \textrm{else}
  
    	\end{array}\right.
		&
		and
		&
			\Delta_A(\lam)[D] \cong
	
			\left \{\begin{array}{c}
        0 \,\,\,\,\,\,\,\, \textrm{if}\, \lam \notin D\\
      	\Delta_A(\lam) \,\,\,\textrm{else}
   	\end{array}\right.
		\end{array}
	\end{displaymath}
	\end{enumerate}
\end{prop}

We now change notation: For the moment graph $\G=(\V, \E, \alpha , \leq)$ associated to the block $\OC_{A,\lam}$ we write $\uparrow$-open for an F-open subgraph. For a subgraph which is F-open according to the moment graph with the reversed order, we write $\downarrow$-open. A subgraph $\mathcal{H}$ with set of vertices $\V'$ is $\uparrow$-open if and only if $\mathcal{H}^c$ is $\downarrow$-open, where $\mathcal{H}^c$ is the full subgraph of $\G$ with vertices $\V^c:= \V - \V'$. For $\mathcal{H}$ an $\downarrow$-open sugraph, set $D$ equal to the set of all $\nu \in \h^*$, such that there exists $x \in \V^c$ with $\nu \leq x \cdot \lam$.\\
Set $O^{\mathcal{H}^c} M := O^D M$ and $M[\mathcal{H}^c] =M [D]$.

\begin{lemma} \label{A}
Let $M$ be a tilting module. There is a natural isomorphism
$$e_{\mathcal{H}} \VV M \cong \VV(d(O^{\mathcal{H}^c} M))$$
\end{lemma}

\begin{proof}
Dualising the short exact sequence
$$O^{\mathcal{H}^c} M \hookrightarrow M \twoheadrightarrow M[\mathcal{H}^c]$$

we get by self-duality of $M$ and the freeness of all modules over $A$ a short exact sequence

$$d(O^{\mathcal{H}^c} M) \twoheadleftarrow M \hookleftarrow d(M[\mathcal{H}^c])$$

Since $\VV$ is exact we get a short exact sequence

$$\VV(d(O^{\mathcal{H}^c} M)) \twoheadleftarrow \VV(M) \hookleftarrow \VV(d(M[\mathcal{H}^c]))$$

Since $d(M[\mathcal{H}^c])\otimes_A Q$ is the direct sum of Verma modules of the form $\Delta_Q (x \cdot \lam)\cong \nabla_Q (x \cdot \lam)$ with $x \notin \V'$ we get $e_{\mathcal{H}} \VV M \cong \VV(d(O^{\mathcal{H}^c} M))$. 
\end{proof}

\begin{definition}
A sheaf $\M$ on $\G$ is called $\downarrow$-flabby (resp. $\downarrow$-projective) if it is F-flabby (resp. F-projective) according to the moment graph $\G$ with reversed order.
\end{definition}

\begin{prop} \label{D}
Let $M$ be a tilting module. Then $\LL(\VV M)$ is $\downarrow$-flabby.
\end{prop}

\begin{proof}
For any $\downarrow$-open subgraph $\mathcal{H}$ we get that $e_{\mathcal{H}}M$ is a free $A$-module by lemma \ref{A}. Now proposition \ref{B} tells us that $\LL(\VV M)$ is $\downarrow$-flabby.
\end{proof}

\begin{notation}
For $x \in \V$ we denote by $\mathscr{B}^{\uparrow}(x)$ the BMP-sheaf $\mathscr{B}(x)$ for our moment graph with the original order. For the moment graph with reversed order we denote the BMP-sheaf at the vertex $x$ by $\mathscr{B}^{\downarrow}(x)$.
\end{notation}

In the following we want to show that the indecomposable deformed tilting module $K_A(x \cdot \lam)$ with highest weight $x\cdot \lam$ corresponds to $\mathscr{B}^{\downarrow}(x)\otimes_S A$ under $\LL \circ \VV$. For this we need some preparation.

\begin{lemma}(\cite{7}, Lemma 7.4.)\label{inj}
Let $M \in \OC$ admit a Verma flag, $\mu \in \h^*$ and $k \in \N$. If $g: (\Delta_\C(\mu))^k \rightarrow M$ is a morphism which induces an injective map $g_\mu : (\Delta_\C(\mu))_\mu ^k \rightarrow M_\mu$ on the $\mu$-weight spaces, then $g$ is injective.
\end{lemma}

\begin{lemma} \label{C}
Let $K \in \OC_{A,\lam}$ be a tilting module. Then for any $w \in \V$ $\LL(\VV K)_w$ is free over $A$ of rank $r:= (K: \Delta_A (w\cdot \lam))$.
\end{lemma}

\begin{proof}
We want to construct an isomorphism $e_w (\VV K) = \LL(\VV K)_w \xrightarrow{\sim} \VV(\nabla_A (w\cdot \lam)^r)$. 
Using corollary \ref{Soe2} and the fact that over the residue field $\K$ we have $\mathrm{dim}_\K \Hom_{\OC_\K}(\Delta_\K(w \cdot \lam) , K \otimes_A \K) = \mathrm{dim}_\K \Hom_{\OC_\K}(K \otimes_A \K, \nabla_\K(w\cdot \lam)) = (K\otimes_A \K : \Delta_\K(w\cdot \lam))$, we can deduce that $\Hom_{\OC_A}(\Delta_A(w\cdot \lam),K)$ is a free $A$- module of rank $r$. Now choose an $A$-basis $f_1,...,f_r$ of $\Hom_{\OC_A}(\Delta_A(w\cdot \lam),K)$. Dualising yields a basis $df_1,..., df_r$ of $\Hom_{\OC_A}(K,\nabla_A(w\cdot \lam))$. Consider now the map
\begin{displaymath}
		\begin{array}{cc}		
		f: \Delta_A(w\cdot \lam)^r \longrightarrow K
		&
		(v_1,...,v_r) \mapsto \sum f_i(v_i)
		\end{array}
\end{displaymath}

Dualising this map yields by self-duality of $K$	

\begin{displaymath}
		\begin{array}{cc}		
		df: K \longrightarrow \nabla_A(w\cdot \lam)^r 
		&
		v \mapsto \sum (df_i(v))
		\end{array}
\end{displaymath}

Applying the functor $\VV$ gives a map
$$\VV df :\VV K \longrightarrow \VV \nabla_A(w\cdot \lam)^r$$
which factors over $e_{w}(\VV K)$, since $\VV \nabla_A (w \cdot \lam)^r$ has support $\{w\}$. So this gives a well-defined map
\begin{displaymath}
		\begin{array}{cc}		
		(\VV df)^w: e_w \VV K \longrightarrow \VV\nabla_A(w\cdot \lam)^r
		
	\end{array}
\end{displaymath}
which is injective, since it becomes an isomorphism after applying $\cdot \otimes_A Q$. We now want to prove that $(\VV df)^w$ is surjective. For this, by Nakayama's lemma and by exactness of $\VV$, it is enough to prove that 
\begin{displaymath}
		\begin{array}{cc}		
		 df\otimes_A \mathrm{id_\C}: K \otimes_A \C \longrightarrow \nabla_A(w\cdot \lam)^r \otimes_A \C
		
	\end{array}
\end{displaymath}
is surjective. But since $f: \Delta_\C (w \cdot \lam)^r \rightarrow K\otimes_A \C$ is injective, by lemma \ref{inj} we get the surjectivity of $df: K\otimes_A \C \rightarrow \nabla_\C(w \cdot \lam)^r$. Since $\VV \nabla_A (w\cdot \lam)^r$ is free of rank $r$ over $A$ we get the result.
\end{proof}

\begin{thm}
Let $K \in \OC_{A,\lam}$ be a tilting module. Then $\LL(\VV K)$ is $\downarrow$-projective as an $A$-sheaf on $\G$.
\end{thm}

\begin{proof}
By \cite{10} section 2.12.(A) $\LL(\VV K)$ is generated by global sections and $\downarrow$-flabby by proposition \ref{D}. Furthermore, the lemma above shows that $(\LL (\VV K))_w$ is free over $A$ for every $w \in \V$.\\
So we only have to prove for any edge $E:=\{x,y\}$ $x> y$ and $\alpha(E)= \check{\alpha}$ where $x,y \in \mathcal{W}_\lam / \mathcal{W}'$ and $\check{\alpha}$ the coroot of a root $\alpha \in \mathcal{R}_\lam$, that  the map $\rho_{x,E} : \LL (\VV K)_x \rightarrow \LL (\VV K)_E$ induces an isomorphism 
$$\LL (\VV K)_x / \check{\alpha} \LL (\VV K)_x \stackrel{\sim}\longrightarrow \LL (\VV K)_E$$

Since $\check{\alpha} \LL (\VV K)_x \subset \mathrm{ker} \rho_{x,E}$, we have to show $\mathrm{ker} \rho_{x,E} \subset \check{\alpha} \LL (\VV K)_x$. For this it suffices to show that
\begin{equation}\label{E}
\mathrm{ker} \rho_{x,E} \subset \check{\alpha} (\LL (\VV K)_x \otimes_A A_\p) = \check{\alpha} \cdot e_x(\LL (\VV K) \otimes_A A_\p)
\end{equation}
for every prime ideal $\p \subset A$ of hight 1. For $\check{\alpha} \notin \p$ (\ref{E}) follows since $\check{\alpha}$ is invertible in $A_\p$.\\
So we have to prove (\ref{E}) for $\p=A \check{\alpha}$.

Since $\rho_{x,E}$ is a push-out map, we can identify $\mathrm{ker} \rho_{x,E}$ with the set $\{e_x u | u \in \VV K(E), \, e_yu =0\}$. Since an element $u \in \VV K(E)$ is of the form $u = (e_x +e_y) v + \check{\alpha} e_x w$ for $v,w \in \VV K$, we have to prove
$$\{e_x u| u \in (e_x+e_y) \VV K, \, e_y u =0\} \subset \check{\alpha} e_x (\VV K \otimes_A A_\p)$$
Since $x > y$, we get $K_{A_\p}(y \cdot \lam) \cong \Delta_{A_\p}(y \cdot \lam)$ and $K_{A_\p}(x \cdot \lam)\cong P_{A_\p}( y \cdot \lam)$. Now we can identify $(e_x +e_y) (\VV K \otimes_A A_\p)$ with a direct sum of $A_\p$-modules of the form $M:= \Hom_{\g_{A_\p}}(P_{A_\p}(y\cdot \lam)  ,  \Delta_{A_\p}(y \cdot \lam)) \cong A_\p$ and $N:= \Hom_{\g_{A_\p}}(P_{A_\p}(y\cdot \lam)  , K_{A_\p}(x \cdot\lam)) \cong A_\p (e_x+e_y) + A_\p \check{\alpha} e_x$ by proposition \ref{Fie5}.\\
If $f \in M$, we get $e_x f =0$ which is equal to $\check{\alpha} e_x f$. If $f \in N$, $e_y f =0$ implies $e_xf \in \check{\alpha} A_\p e_x$. But with our identification we get $e_x f \in \check{\alpha} e_x (\VV K \otimes_A A_\p)$.
\end{proof}

\begin{cor}
We have $\LL (\VV K_A (w \cdot \lam)) \cong \mathscr{B}^\downarrow(w) \otimes_S A$ for all $w \in \mathcal{W}_\lam$.
\end{cor}

\begin{proof}
The proof is essentially the same as in \cite{10} Theorem 4.22. and relies on the facts that $\LL (\VV K_A (w \cdot \lam))$ is $\downarrow$-projective and indecomposable. Now the description of the indecomposable $\downarrow$-projective sheaves by BMP-sheaves gives the claim.
\end{proof}

\begin{cor}
We have $(K(w \cdot \lam):\Delta(x \cdot \lam))= \mathrm{rk}_S \mathscr{B}^\downarrow(w)_x$.
\end{cor}

\begin{proof}
Lemma \ref{C} shows that $\textrm{rk}_A (\LL (\VV K_A (w \cdot \lam)_x) = (K_A(w \cdot \lam):\Delta_A(x \cdot \lam))= (K(w \cdot \lam):\Delta(x \cdot \lam))$. Now apply the above corollary.
\end{proof}

Let $w_{\circ} \in \mathcal{W}_\lam$ be the longest element according to the Bruhat order. Then the multiplication from the left
$$w_{\circ}: \G \longrightarrow \G^{\circ}, \,\,\, x\mathcal{W}' \mapsto w_{\circ} x \mathcal{W}'$$
is an isomorphism of moment graphs in the sense of \cite{12}. Thus, it induces a pull-back functor (\cite{12} definition 3.3.)

$$w_{\circ}^*: \mathcal{SH}_A(\G^{\circ})^f \longrightarrow \mathcal{SH}_A(\G)^f$$

\begin{prop}
$w_{\circ} ^*(\mathscr{B}^\downarrow (x)) = \mathscr{B}^\uparrow (w_{\circ}x)$
\end{prop}

\begin{proof}
\cite{12} Lemma 5.1.
\end{proof}

\begin{cor}
$(K(w \cdot \lam):\Delta(x \cdot \lam))= P_{x,w}(1)$ where $P_{x,w}$ denotes the Kazhdan-Lusztig polynomial.
\end{cor}

\begin{proof}
Since $w_{\circ}^* \mathscr{B}^\downarrow(w) =\mathscr{B}^\uparrow(w_{\circ}w)$ we get $\mathrm{rk}_S \mathscr{B}^\downarrow(w)_x =\mathrm{rk}_S \mathscr{B}^\uparrow(w_{\circ} w)_{w_\circ x}$ and the result follows since the stalks of BMP-sheaves describe the local equivariant intersection cohomology of the according Schubert variety by \cite{3} Theorem 1.6.
\end{proof}

\begin{remark}
Character formulae for tilting modules were already discovered in \cite{A} for Kac-Moody algebras by using a tilting equivalence on category $\OC$ which interchanges projective with tilting modules. Our approach, however, uses sheaves on moment graphs but only works for the finite dimensional case.
\end{remark}

\section{The Jantzen and Andersen filtrations}

We fix a deformed tilting module $K \in \mathcal{K}_A$ and let $\lam \in \h^{*}$. The composition of homomorphisms induces an $A$-bilinear pairing
\begin{eqnarray*}
	\begin{array}{ccc}
	\Hom_{\OA}(\Delta_A(\lam),K)\times \Hom_{\OA}(K,\nabla_A(\lam)) & \longrightarrow & \Hom_{\OA}(\Delta_A(\lam),\nabla_A(\lam))\cong A\\
	\\
	(\varphi,\psi) & \longmapsto & \psi \circ \varphi
	\end{array}
\end{eqnarray*}

For any $A$-module $H$ we denote by $H^{*}$ the $A$-module $\Hom_A(H,A)$. As in \cite{14} Section 4 one shows that for $A$ a localization of $S$ at a prime ideal $\p$ our pairing is non-degenerated and induces an injective map
$$E=E_A^{\lam}(K): \Hom_{\OA}(\Delta_A(\lam),K) \longrightarrow \left(\Hom_{\OA}(K,\nabla_A(\lam))\right)^{*}$$
of finitely generated free $A$-modules.\\
If we take $A=\C \llbracket v \rrbracket$ the ring of formal power series around the origin on a line $\C\delta \subset \h^{*}$ not contained in any hyper plane corresponding to a reflection of the Weyl group (e.g. $\C \rho$, where $\rho$ is the half sum of positive roots), we get a filtration on $\Hom_{\OA}(\Delta_A(\lam),K)$ by taking the preimages of $\left (\Hom_{\OA}(K,\nabla_A(\lam)) \right )^{*}\cdot v^{i}$ for $i=0,1,2,...$ under $E$.

\begin{definition}[\cite{14}, Definition 4.2.]
Given $K_\C \in \mathcal{K}_\C$ a tilting module of $\OC$ and $K \in \mathcal{K}_{\C \llbracket v \rrbracket}$ a preimage of $K_\C$ under the functor $\cdot \otimes_{\C \llbracket v \rrbracket} \C$, which is possible by Proposition \ref{tilt1} with $S\rightarrow \C \llbracket v \rrbracket$ the restriction to a formal neighborhood of the origin in the line $\C \rho$. Then the image of the filtration defined above under specialization $\cdot \otimes_{\C \llbracket v \rrbracket} \C$ is called the \textit{Andersen filtration} on $\Hom_\g (\Delta(\lam), K_\C)$.
\end{definition}

The Jantzen filtration on a Verma module $\Delta(\lam)$ induces a filtration on the vector space $\Hom_\g(P , \Delta(\lam))$, where $P$ is a projective object in $\OC$. Now consider the embedding $\Delta_{\C \llbracket v \rrbracket}(\lam) \hookrightarrow \nabla_{\C \llbracket v \rrbracket}(\lam)$. Let $P_{\C \llbracket v \rrbracket}$ denote the up to isomorphism unique projective object in $\OC_{\C \llbracket v \rrbracket}$ that maps to $P$ under $\cdot \otimes_{\C \llbracket v \rrbracket}\C$, which is possible by theorem \ref{Fie1}. We then get the same filtration if we first take the preimages of $\Hom_{\OC_{\C \llbracket v \rrbracket}}(P_{\C \llbracket v \rrbracket},\nabla_{\C \llbracket v \rrbracket}(\lam)) \cdot v^{i}$, $i=0,1,2,...$, under the induced inclusion
$$J=J_A^{\lam}(P):\Hom_{\OA}(P_A,\Delta_A(\lam)) \longrightarrow \Hom_{\OA}(P_A,\nabla_A(\lam))$$
for $A=\C \llbracket v \rrbracket$ and then the image of this filtration under the map\\
 $\Hom_{\OA}(P_A,\Delta_A(\lam))\twoheadrightarrow \Hom_{\g}(P,\Delta(\lam))$ induced by $\cdot \otimes_A \C$.

\subsection{Sheaves with a Verma flag}

For $M \in \ZZ_A-\mathrm{mod}^f$ and $\mathcal{I} \subset \V$, we define
$$M_\mathcal{I} := M \cap \bigoplus_{x\in \mathcal{I}} e_x (M \otimes_A Q)$$
and 
$$M^{\mathcal{I}}:= \mathrm{Im} \left( M \rightarrow M\otimes_A Q \rightarrow \bigoplus_{x\in \mathcal{I}} e_x (M \otimes_A Q) \right)$$

\begin{definition}
We say that $M \in \ZZ_A -\mathrm{mod}^f$ admits a Verma flag if the module $M^{\mathcal{I}}$ is free (graded free in case $A=S$) for each F-open subset $\mathcal{I}$.
\end{definition}
Denote the full subcategory of $\ZZ_A(\G)-\mathrm{mod}^f$ consisting of all modules admitting a Verma flag by $\ZZ_A(\G)-\mathrm{mod}^{VF}$. Now for any vertex $x\in \V$ and an $A$-sheaf $\M \in \mathcal{SH}_A(\G)^f$ define
$$\M^{[x]}:= \mathrm{ker}(\M_x \rightarrow \bigoplus_{E \in U_x} \M_E)$$
where $U_x$ is from Notation \ref{not1}. Furthermore, denote by 
$$\M^{x}:= \bigcap_E \mathrm{ker}(\rho_{x,E})$$
the \textit{costalk} of $\M$ at the vertex $x$. Here the intersection runs over all edges $E \in \E$ with $x \in E$.\\
Denote the image of $\ZZ_A(\G)-\mathrm{mod}^{VF}$ under $\LL$ by $\mathcal{C}_A(\G)$. The next proposition gives an explicit description of $\mathcal{C}_A(\G)$

\begin{prop}(\cite{C}, Proposition 2.9)
For $M \in \ZZ_A(\G)-\mathrm{mod}^f$, set $\M= \LL(M)$. Then $M$ admits a Verma flag if and only if $\M$ is flabby and $\M^{[x]}$ is (graded) free for all $x \in \V$.
\end{prop} 

In \cite{8} section 2.6 Fiebig introduces a duality $D$ on $\ZZ_S-\mathrm{mod}^f$. For $M \in \ZZ_S-\mathrm{mod}^f$ we set 
$$D(M)= \bigoplus_{i\in \Z} \Hom_S^i(M,S)$$
where $\Hom_S^i(M,S) = \Hom_S(M,S[i])$ ($[\cdot]$ grading shift). This induces an equivalence $D : \mathcal{C}_S(\G) \stackrel{\sim}{\rightarrow} \mathcal{C}_S^{op}(\G^{\circ})$, where $\G^{\circ}$ denotes the moment graph $\G$ with reversed order.\\
For $A$ a localization of $S$ at a prime ideal we define $D_A:= \Hom_A(\cdot, A): \mathcal{C}_A(\G) \stackrel{\sim}{\rightarrow} \mathcal{C}_A^{op}(\G^{\circ})$. 

\begin{thm}(\cite{8}, Theorem 6.1 and \cite {D}, Proposition 3.20)
The BMP-sheaves are self-dual: $D \mathscr{B}^\uparrow (x) \cong \mathscr{B}^\uparrow (x)$ for all $x\in \V$.
\end{thm}

Recall the pull-back functor of section \ref{BMP}
$$w_{\circ}^*: \mathcal{SH}_A(\G^{\circ})^f \longrightarrow \mathcal{SH}_A(\G)^f$$

\begin{lemma}
We have $w_{\circ}^*(\mathcal{C}_A(\G^{\circ}))=\mathcal{C}_A(\G)$.
\end{lemma}

\begin{proof}
Let $\M \in \mathcal{C}_A(\G^{\circ})$. We have to show that $w_{\circ}^*(\M)$ is $\uparrow$-flabby and $(w_{\circ}^*(\M))^{[x]}$ is free over $A$ for $x\in \V$.\\
Let $\mathcal{I}$ be $\downarrow$-open. Then $w_{\circ}\mathcal{I}$ is $\uparrow$-open and we get that
$$\Gamma(\M)\cong \Gamma(w_{\circ}^*(\M)) \rightarrow \Gamma(\mathcal{I}, w_{\circ}^*(\M)) \cong \Gamma(w_{\circ} \mathcal{I}, \M)$$
is surjective since $\M$ is flabby.\\
As $(w_{\circ}^*(\M))^{[x]}=\M^{[w_{\circ}x]}$ the claim follows.
\end{proof}

We get an equivalence of categories
$$w_\circ ^*:\mathcal{C}_A(\G^{\circ}) \longrightarrow \mathcal{C}_A(\G)$$
with the properties: $w_\circ ^*(\mathscr{B}^\uparrow (x)\otimes_S A) \cong \mathscr{B}^\downarrow (w_\circ x)\otimes_S A$ and $w_\circ ^*(\mathscr{V}_A (x)) \cong \mathscr{V}_A (w_\circ x)$. Thus the composition 
$$F_A= (w_\circ ^*)^{op} \circ D_A : \mathcal{C}_A(\G) \longrightarrow \mathcal{C}_A(\G)^{op}$$
is an equivalence with $F_A(\mathscr{B}^\uparrow (x)\otimes_S A) \cong \mathscr{B}^\downarrow (w_\circ x)\otimes_S A$ and $F(\mathscr{V}_A (x)) \cong \mathscr{V}_A (w_\circ x)$ which are isomorphisms of graded sheaves if $A=S$.\\

\begin{thm}(\cite{7}, Theorem 7.1.)
The functor $\VV : \mathcal{M}_{A,\lam}   \rightarrow \mathcal{C}_A(\G)$ is an equivalence of categories for $A=S_{(0)}$.
\end{thm}

Now we can lift the functor $F_A$ via Fiebig's equivalence to a functor $T_A$ on the representation theoretic side such that the following diagram of functors commutes:

\begin{eqnarray*}
	\begin{CD}
   \mathcal{M}_{A,\lam}   @> \VV >> \mathcal{C}_A(\G)\\
   @VT_A VV @VVF_A  V\\
    \mathcal{M}^{op}_{A,\lam} @> \VV ^{op} >> \mathcal{C}_A(\G)^{op}
	\end{CD}
\end{eqnarray*}

\begin{thm}
Let $\lam\in \h^{*}$ be antidominant, $x,y \in \W$ and $w_\circ \in \W$ the longest element. Denote by $A=S_{(0)}$ the localization of $S$ at $0$. There exists an isomorphism $L=L_A(x,y)$ which makes the diagram

{\small
\begin{eqnarray*}\label{eqn:DiaS}
	\begin{CD}
   \Hom_{\OC_A}(P_A(x \cdot \lam),\Delta_A(y \cdot \lam))   @>J>> \Hom_{\OC_A}(P_A(x \cdot \lam),\nabla_A(y \cdot \lam))\\
   @VVT V @VVL V\\
   \Hom_{\OC_A}(\Delta_A(w_\circ y \cdot \lam),K_A(w_\circ x \cdot \lam)) @>E >> \left(\Hom_{\OC_A}(K_A(w_\circ x \cdot \lam),\nabla_A(w_\circ y \cdot \lam))\right)^{*}
	\end{CD}
\end{eqnarray*}}
commutative. Here $J=J_A^{y\cdot \lam}(P_A(x \cdot\lam))$ and $E=E_A^{w_\circ y \cdot \lam}(K_A(w_\circ x \cdot \lam))$ denote the inclusions defined above and $T=T_A$ denotes the isomorphism induced by the functor $T_A$ from above.
\end{thm}

\begin{proof}
The proof is essentially the same as the one for Theorem 4.2. in \cite{11}, where we prove a similar result for $T_A$ the tilting functor.
\end{proof}
Denote by $T_\C:\Hom_\g(P(x \cdot \lam),\Delta(y \cdot \lam)) \stackrel{\sim}{\rightarrow} \Hom_\g(\Delta(w_\circ y \cdot \lam),K(w_\circ x \cdot \lam))$ the isomorphism we get from $T_A \otimes_A \mathrm{id}_\C$ after base change. The next corollary now follows in the same way as Corollary 4.3. in \cite{11}.
\begin{cor}
The isomorphism 
$$T_\C:\Hom_\g(P(x \cdot \lam),\Delta(y \cdot \lam)) \stackrel{\sim}{\rightarrow} \Hom_\g(\Delta(w_\circ y \cdot \lam),K(w_\circ x \cdot \lam))$$
identifies the filtration induced by the Jantzen filtration with the Andersen filtration.
\end{cor}

Now we consider $\C \cong S/S\h$ as a simple graded $S$-module living in degree $0$. The map 
$$T_\C:\Hom_\g(P(x \cdot \lam),\Delta(y \cdot \lam)) \stackrel{\sim}{\rightarrow} \Hom_\g(\Delta(w_\circ y \cdot \lam),K(w_\circ x \cdot \lam))$$
can then be identified with 
$$F_S \otimes \mathrm{id}_\C: \Hom_{\mathcal{C}_S(\G)}(\mathscr{B}^\uparrow (x), \mathscr{V}_S (y))\otimes_S \C \stackrel{\sim}{\rightarrow} \Hom_{\mathcal{C}_S(\G)} (\mathscr{V}_S (w_{\circ} y),\mathscr{B}^\downarrow (w_{\circ} x))\otimes_S \C$$
which is now an isomorphism of graded $\C$-vector spaces. But using the proof of proposition 7.1. (3) in \cite{6} this isomorphism becomes a graded isomorphism between certain costalks of the Braden-MacPherson sheaves, namely an isomorphism
$$\varphi: \mathscr{B}^\uparrow (x)^y \otimes_S \C \stackrel{\sim}{\rightarrow} \mathscr{B}^\downarrow (w_{\circ} x)^{w_{\circ}y}\otimes_S \C$$

In \cite{14} Soergel shows that the filtration on $\mathscr{B}^\downarrow (w_{\circ} x)^{w_{\circ}y}\otimes_S \C$ induced by the Andersen filtration coincides with the grading filtration we get from the grading on the Braden-MacPherson sheaf $\mathscr{B}^\downarrow (w_{\circ} x)$. Since the graded isomorphism $\varphi$ interchanges the filtration on $\mathscr{B}^\uparrow (x)^y \otimes_S \C$ induced by the Jantzen filtration with the filtration on $\mathscr{B}^\downarrow (w_{\circ} x)^{w_{\circ}y}\otimes_S \C$ induced by the Andersen filtration, we get that the Jantzen filtration coincides with the grading filtration coming from the grading on the Braden-MacPherson sheaf $\mathscr{B}^\uparrow (x)$.

\section*{Acknowledgements}

I would like to thank my supervisor Peter Fiebig for many helpful discussions and for sharing his deep insight in the theory of sheaves on moment graphs.

\bibliographystyle{amsplain}

\end{document}